\documentclass[hidelinks]{amsart}
\usepackage[T1]{fontenc}
\usepackage{hyperref}
\usepackage{paralist}
\usepackage{amsfonts, amsmath, amstext, amssymb, amsthm, graphicx, setspace, fullpage, mathtools}
\usepackage{thmtools}
\usepackage{multirow}
\usepackage{xcolor}
\usepackage{pifont} % for tag
\usepackage{tikz-cd}
\usepackage{cancel}

\newcommand{\R}{\mathbb{R}}
\newcommand{\Z}{\mathbb{Z}}

\newcommand{\Id}{\text{Id}}

\newcommand{\Isom}{\text{Isom}}

\newcommand{\hra}{\hookrightarrow}
\newcommand{\thra}{\twoheadrightarrow}
\newcommand{\xra}{\xrightarrow}

\newcommand{\Lra}{\Leftrightarrow}
\newcommand{\Ra}{\Rightarrow}

\theoremstyle{plain}
\newtheorem{thm}{Theorem}[section] %reset the counter at every new section
\newtheorem{lem}[thm]{Lemma}
\newtheorem{prop}[thm]{Proposition}
\newtheorem{cor}[thm]{Corollary}
\newtheorem{rem}[thm]{Remark}
\theoremstyle{definition}
\newtheorem{claim}{Claim}
\newtheorem{defn}{Definition}

\DeclarePairedDelimiter\abs{\lvert}{\rvert}
\DeclarePairedDelimiterX{\norm}[1]{\lVert}{\rVert}{#1}

\title{Simplicial Volume of Closed Locally Homogeneous Riemmannian Manifolds}
\author{Peng Hui How}
\address{Department of Mathematics, University of Chicago}
\email{hph@uchicago.edu}

\begin{document}
\maketitle

\begin{abstract}
We show (in Theorem~\ref{t_sv}) that every closed, locally homogeneous Riemannian manifold with positive simplicial volume must be homeomorphic to a locally symmetric space of non-compact type, giving a converse to \cite{lafont-schmidt:2006} within the scope of closed, locally homogeneous Riemannian manifolds.
This characterizes all closed locally homogeneous Riemannian manifolds with nonzero simplicial volume.
\end{abstract}

%\tableofcontents
\section{Introduction}
In \cite{gromov:1982}, Gromov introduced the following non-negative topological invariant for a connected, closed manifold:
\begin{defn}
Let $M$ be a connected, closed, $n$-manifold.

If $M$ is oriented, the \emph{simplicial volume} of $M$ is the $L^1$-norm of its fundamental class $[M] \in H_n(M;\R)$
\[
	\norm{M}
	= \inf_{\substack{c \in C_k(M;\R)\\ \partial{c}=0 \\{[c]}=[M]}}\norm{c}_1
\]
where $\norm{c}_1$ is the $L^1$-norm on $C_k(M;\R)$, i.e.
\[
  \norm{\sum_ic_i\sigma_i}_1 = \sum_i\abs{c_i},
\]
where $c_i \in \R$, $\sigma_i \in C(\Delta^k,M)$.

If $M$ is not oriented, we define $\norm{M}$ to be $\frac{1}{2}\norm{\overline{M}}$ where $\overline{M}$ is its orientation double cover.
\end{defn}

Before we proceed with our main result, we recall two definitions, in their bare minimal forms.
For a more conceptual and comprehensive formulation, see \cite{helgason:1979} and {\cite[\S~2.2]{eberlein:1996}}.
\begin{defn}
	Let $M$ be a complete connected Riemannian manifold.
	\begin{enumerate}
		\item We say that $M$ is \emph{locally homogeneous} if for every $p, q$ in $M$, there exists open neighborhoods $U \ni p$, $V \ni q$ and $f \in \Isom(U,V)$ such that
		$\begin{cases}
			f(U) = V
			\\
			f(p) = q
		\end{cases}$.
		\item We say that $M$ is a \emph{locally symmetric space of non-compact type} if it is diffeomorphic to $\Gamma \backslash G / K$ where $G$ is a centerless, semisimple Lie group, $K$ is a maximal compact subgroup in $G$, and $\Gamma$ is a discrete subgroup in $G$ that acts freely on $G/K$ via left action.
		In which case, the universal cover of $M$ is $G/K$.
		Note that such $M$ is locally homogeneous.
	\end{enumerate}
\end{defn}
\begin{rem}
	Since $\norm{M}$ is a topological invariant of $M$, to answer whether or not $\norm{M} = 0$, it suffices to analyze its topological type.
	In particular, it suffice to analyze its diffeomorphism type.
	In Lemma~\ref{lhm_as_double-coset-manifold}, we show that $M$ is diffeomorphic to a double coset manifold $\Gamma \backslash G / K$ where
	\begin{itemize}
		\item $G$ is a Lie group with finitely many connected components.
		\item $K$ is a compact subgroup of $G$.
		\item $\Gamma$ is a discrete subgroup of $G$ that acts freely on $G/K$.
	\end{itemize}
	Conversely, for any $(G,K,\Gamma)$ that satisfies the above conditions, since $K$ is compact, by basic Lie theory, $\Gamma \backslash G/K$ can be equipped with a locally homogeneous Riemannian metric.
	This allows us to treat locally homogeneous Riemannian manifolds as double coset manifolds interchangably.
\end{rem}

Let $M$ be a closed locally homogeneous Riemannian manifold.
A criterion for $\norm{M} > 0$ is of natural interest.
In \cite{lafont-schmidt:2006}, applying estimates of Connell-Farb in \cite{connell-farb:2001} on the barycenter map in higher rank, Lafont-Schmidt showed that every closed, locally symmetric space of non-compact type has positive simplicial volume.
In this paper, we prove the converse:
\begin{thm}\label{t_sv}
%Let $M$ be a closed, locally homogeneous Riemannian manifold.
Let $M$ be a closed topological manifold that admits a smooth structure that admits a locally homogeneous Riemannian metric.
Then $\norm{M} > 0$ if and only if $M$ is a closed, locally symmetric space of non-compact type.
\end{thm}

A notion closely related to local homogeneity is Thurston model geometry; see \cite[\S3]{thurston:1997}.
In this language, Theorem~\ref{t_sv} immediately implies the following:
\begin{cor}\label{t_sv_cor}
	Let $M$ be a closed manifold modeled on some Thurston geometry.
	Then $\norm{M} > 0$ if and only if $M$ is homeomorphic to a closed locally symmetric space of non-compact type.
\end{cor}
\begin{rem}
	Let $n = \dim(M)$.
	The special cases of $n = 2,3$ of Corollary~\ref{t_sv_cor} were essentially solved by Gromov.
	The special case of $n = 4$ (resp. $n = 5$) was done by Zhang in \cite{zhang:2017} (resp. Neofytidis-Zhang in \cite{neofytidis-zhang:2021}), based on work on classification of dim-4 (resp. 5) geometry by Filipkiewicz (resp. Geng).
\end{rem}

The proof of Theorem~\ref{t_sv} relies heavily on the following structure theorem of closed, locally homogeneous Riemannian manifolds, which we will prove.
Before we state the theorem, we give the following definition, which will be used repeatedly in this paper.
\begin{defn}
We say that a group $G$ is \emph{$X$-by-$Y$} if there exists a short exact sequence
\[
	1 \to N \to G \to Q \to 1
\]
where $N$ satisfies property $X$ and $Q$ satisfies property $Y$.
For example, $G$ is said to be solvable-by-compact if there exists a short exact sequence as above where $N$ is solvable and $Q$ is compact.
\end{defn}
\begin{rem}\label{group-by-extension-remark}
	The binary operation given by "-by-" need not be associative, i.e. a ($X$-by-$Y$)-by-$Z$ group need a not be a $X$-by-($Y$-by-$Z$) group.
	However, in the special case when $G$ is the kernel of the composition of a sequence of surjective group homomorphisms, i.e.
	\[
		G = \ker(H \stackrel{p_1}{\thra} H_1 \stackrel{p_1}{\thra} H_2 \stackrel{p_1}{\thra} H_3),
	\]
	where $\ker(p_1)$ (resp. $\ker(p_2)$, $\ker(p_3)$) satisfies property $X$ (resp. $Y$, $Z$),
	then $G = \ker(p_3 \circ p_2 \circ p_1) = \ker(p_3 \circ (p_2 \circ p_1)) = \ker((p_3 \circ p_2) \circ p_1)$, and thus the both following short exact sequences hold:
	\begin{enumerate}
		\item $1 \to \ker(p_2 \circ p_1) \to G \to \ker(p_3) \to 1$, which implies that $G$ is ($X$-by-$Y$)-by-$Z$.
		\item $1 \to \ker(p_1) \to G \to \ker(p_3 \circ p_2) \to 1$, which implies that $G$ is $X$-by-($Y$-by-$Z$).
	\end{enumerate}
	In this case, since dropping the parenthesis does not introduce ambiguity, we would simply denote $G$ as a \emph{$X$-by-$Y$-by-$Z$} group.
\end{rem}

We are now ready to state the structure theorem for closed, locally homogeneous Riemannian manifolds.
\begin{thm}[Structure theorem for closed, locally homogeneous Riemannian manifolds]\label{lhm}
	%Let $M$ be a closed locally homogeneous Riemannian manifold.
	Let $M$ be a closed topological manifold that admits a smooth structure that admits a locally homogeneous Riemannian metric.
	Then $M$ admits a finite cover $M'$ such that there exists a fiber bundle
	\[
		F \to M' \to B
	\]
	where:
	\begin{enumerate}
		\item $F$ is a closed locally homogeneous Riemannian manifold $\Gamma \backslash G / K$ given by a (solvable-by-compact-by-abelian)-by-compact Lie group $G$.
		\item $B$ is a closed locally symmetric space of non-compact type.
	\end{enumerate}
	We denote such $M'$ as a \emph{fibered cover} of $M$.
\end{thm}
\begin{rem}
Theorem~\ref{lhm} is likely known to experts, but I could not find it in this form in the literature.
A similar result (in a more general context, with a longer proof and more indirect specification of the fiber to take care of higher generality) was proved by van Limbeek in \cite{vanLimbeek:2014}.
\end{rem}

\subsection*{Acknowledgement}
The author would like to thank her advisor Benson Farb for suggesting this problem and his helpful discussions, comments, editing of drafts, and numerous encouragement throughout the project.
She would also like to express her utmost gratitude to Wouter van Limbeek for his insightful discussions, indispensible technical help, as well as comments to the earlier drafts.
Last but not least, she is extremely grateful to Clara L\"oh, for her careful reading and various suggestions to an earlier draft of the paper.

\section{Theorem~\ref{lhm} implies Theorem~\ref{t_sv}}
Our goal in this section is to prove Theorem~\ref{t_sv} using Theorem~\ref{lhm}.
An immediate corollary of Theorem~\ref{lhm} is the following, which we will use to prove Theorem~\ref{t_sv}.
\begin{cor}\label{lhm_pi_1_F_amenable}
	Let $M$ be a closed locally homogeneous Riemannian manifold.
	Then $M$ admits a finite cover $M'$ such that there exists a fiber bundle
	\[
		F \to M' \to B
	\]
	where:
	\begin{enumerate}
		\item $F$ is a closed connected (possibly trivial) manifold with amenable $\pi_1(F)$.
		(Recall that a locally compact topological group $G$ is said to be \emph{amenable} if $G$ has a fixed point for every continuous affine action on a nonempty $G$-space which is a compact, convex subset of a locally convex topological vector space.)
		\item $B$ is a closed locally symmetric space of non-compact type.
	\end{enumerate}
	We denote such $M'$ as a \emph{fibered cover} of $M$.
\end{cor}
\begin{proof}
	We use the notations in Theorem~\ref{lhm}.
	The only thing to be proven is that $\pi_1(F)$ is amenable.

	To do so, we recall the following:
	\begin{prop}[{\cite[\S12.1,\S12.2]{morris:2001}}]\label{amenable}
		\begin{enumerate}
			\item The following types of groups are amenable: compact, solvable.
			\item The amenable property is preserved under: group extension, quotients, taking closed subgroups, taking finite-index subgroups.
		\end{enumerate}
	\end{prop}
	Continuing with the proof of Corollary~\ref{lhm_pi_1_F_amenable}, since $G$ is (solvable-by-compact-by-abelian)-by-compact, by Proposition~\ref{amenable}, $G$ is amenable.

	To prove that $\pi_1(F)$ is amenable, note that basic covering space theory yields the short exact sequence of groups
	\[
		1 \to \pi_1(G/K) \to \pi_1(F) \to \Gamma \to 1.
	\]
	Since $\Gamma$ is a closed subgroup of $G$, by Proposition~\ref{amenable}, $\Gamma$ is amenable.
	By Proposition~\ref{amenable}, since the amenable property is closed under group extension, it suffices to show that $\pi_1(G/K)$ is abelian.
	This holds, since the long exact sequence of fibration for $K \to G \to G/K$ yields the short exact sequence of groups
	\[
		1 \to \pi_1(K) \to \pi_1(G) \to \pi_1(G/K) \to 1,
	\]
	By elementary algebraic topology, the fundamental group of any topological group is abelian.
	Thus, in particular, $\pi_1(G)$ is abelian, and thus $\pi_1(G/K)$ is also abelian.
	Since $\pi_1(G/K)$ is abelian, by Proposition~\ref{amenable}, it is amenable, as required.
\end{proof}

Now, still assuming Theorem~\ref{lhm}, which as we just showed implies Corollary~\ref{lhm_pi_1_F_amenable} as a special case, we prove the following, which contains Theorem~\ref{t_sv}.
\begin{lem}\label{lhm_tfae}
	Let $M$ be a closed locally homogeneous Riemannian manifold, and $M'$ a fibered cover (as in Theorem~\ref{lhm}) of $M$.
	Then the following are equivalent:
	\begin{enumerate}[(1)]
		%\item $F$ is trivial (i.e. a single point).
		\item $M'$ is a closed locally symmetric space of non-compact type.
		\item $M$ is a closed locally symmetric space of non-compact type.
		\item $\norm{M'} > 0$.
		\item $\norm{M} > 0$.
	\end{enumerate}
\end{lem}

\begin{proof}
Since $M'$ is a finite cover of $M$, the following holds:
\begin{enumerate}
	\item (1) $\Lra$ (2):
	This is an immediate consequence of the following facts:
	\begin{enumerate}
		\item Their universal covers agree, i.e. $\widetilde{M'} = \widetilde{M}$.
		\item The following are equivalent for a Riemannian manifold $N$:
		\begin{enumerate}
			\item $N$ is a locally symmetric space of non-compact type.
			\item its universal cover $\widetilde{N}$ is a symmetric space of non-compact type.
		\end{enumerate}
		\item The compactness property is preserved under finite-covering.
	\end{enumerate}
	\item (3) $\Lra$ (4):
	This is an immediate consequence of the following remark:
	\begin{rem}[{\cite[\S~0.2]{gromov:1982}}]\label{simplicial_vol_finite_cover}
	Let $N \thra M$ be a finite $d$-sheeted covering.
	Then $\norm{N} = d\norm{M}$.
	In particular, $\norm{N} = 0$ $\Lra$ $\norm{M} = 0$.
	\end{rem}
\end{enumerate}
As stated in the introduction, (1) $\Ra$ (3) was proven in \cite[Main Theorem]{lafont-schmidt:2006}.
It remains to prove (3) $\Ra$ (1).

By Corollary~\ref{lhm_pi_1_F_amenable}, there is a fiber bundle
\[
	F \to M' \to B
\]
where:
\begin{itemize}
	\item $\pi_1(F)$ is amenable.
	\item $B$ is a closed locally symmetric space of non-compact type.
\end{itemize}
According to {\cite[Exercise~14.13]{luck2002l2}},
if there is a fiber bundle
\[
	F \hra E \thra B
\]
with $\pi_1(F)$ amenable and $F$ nontrivial (i.e. $F$ connected with $\dim(F) > 0$), then $\norm{E} = 0$.

This implies that if $\norm{M} > 0$, $F$ must be a point, i.e. $M'$ is a closed locally symmetric space of non-compact type.
This completes the proof of Lemma~\ref{lhm_tfae}.
\end{proof}

\section{The Structure Theorem for Closed, Locally Homogeneous Riemannian Manifolds}
Our goal in this section is to prove Theorem~\ref{lhm} (the structure theorem for closed locally homogeneous Riemannian manifolds).
To start off, it is crucial to note that a locally homogeneous Riemannian manifold is diffeomorphic to a double coset manifold.
\begin{lem}\label{lhm_as_double-coset-manifold}
	Let $M$ be a locally homogeneous Riemannian manifold.
	Then, there exists
	\begin{itemize}
		\item $G$ a Lie group with finitely many connected-components
		\item $K$ a compact subgroup of $G$
		\item $\Gamma$ a discrete subgroup of $G$ that acts freely on $G/K$
	\end{itemize}
	where
	\begin{enumerate}
		\item The universal cover $\widetilde{M}$ of $M$ is diffeomorphic to $G/K$, and
		\item $M$ is diffeomorphic to $\Gamma \backslash G/K$.
	\end{enumerate}
\end{lem}
\begin{proof}
	Let $M$ is a locally homogeneous Riemannian manifold.
	By Singer\cite{singer1960infinitesimally}, its universal cover $\widetilde{M}$, equipped with the pullback metric, is a homogeneous Riemannian manifold.
	Let $G$ be the isometry group of $\widetilde{M}$, and $K$ the isotropy subgroup of any arbitrary point in $\widetilde{M}$. Since $K$ is the image of $\mathfrak{o(\frak{g})}$ under the exponential map, it is compact.

	We claim that $\widetilde{M}$ is diffeomorphic to $G/K$.
	To see this, note that by assumption, $G$ acts transitively on $\widetilde{M}$.
	If $G$ acts on $\widetilde{M}$ on the left (resp. right), then, as Riemannian manifolds, $\widetilde{M} = G/K$ (resp. $\widetilde{M} = K \backslash G$), where the latter is equipped with a Riemannian metric descended from a left-invariant (resp. right-invariant) Riemannian metric on $G$.
	In particular, as smooth manifolds, $\widetilde{M} = G/K$, as desired.

	We claim that $G$ has finitely many connected components.
	To see this, note that since  $\widetilde{M}$ is simply-connected, it follows from the long exact sequence of fibration of $K \to G \to \widetilde{M}$ that $\pi_0(G) \cong \pi_0(K)$.
	Since $K$ is compact, thus $\pi_0(K)$ is finite, and so is $\pi_0(G)$, as desired.

	With above in mind, let $\Gamma = \pi_1(M)$.
	Since it is the deck transformation group of $\widetilde{M} \to M$, which is a local isometry, thus $\pi_1(M)$ is a discrete subgroup of $G$ that acts freely on $\widetilde{M} = G/K$.
	This completes the proof.
\end{proof}

Now, recall that Theorem~\ref{lhm} is a statement that holds up to finite cover.
Thus, with the following (a corollary of Lemma~\ref{lhm_as_double-coset-manifold}), we reduce our consideration to the case when $M$ is a double coset manifold $\Gamma \backslash G / K$ where $G$ is a connected Lie group.
\begin{lem}~\label{finite-cover-double-coset-manifold}
	Let $M$ be a closed locally homogeneous Riemannian manifold.
	Then $M$ has a finite-sheeted cover given by a double coset manifold $\Gamma \backslash G / K$ where
	\begin{itemize}
		\item $G$ is a connected Lie group
		\item $K$ is a compact subgroup of $G$
		\item $\Gamma$ is a cocompact lattice in $G$
	\end{itemize}
\end{lem}
\begin{proof}
	Let $M$ be a closed locally homogeneous Riemannian manifold.
	By Lemma~\ref{lhm_as_double-coset-manifold}, there exists
	\begin{itemize}
		\item $G'$ a Lie group with finitely many connected-components
		\item $K'$ a compact subgroup of $G$
		\item $\Gamma'$ a discrete subgroup of $G'$ that acts freely on $G'/K'$
	\end{itemize}
	where
	\begin{enumerate}
		\item The universal cover $\widetilde{M}$ of $M$ is diffeomorphic to $G'/K'$, and
		\item $M$ is diffeomorphic to $\Gamma' \backslash G'/K'$.
	\end{enumerate}

	With above in mind, let
	\begin{itemize}
		\item $G$ be the identity connected component of $G'$.
		\item $K = K' \cap G$ ($K$ is a compact subgroup in $G$).
		\item $\Gamma = \Gamma' \cap G$.
	\end{itemize}
	Since $G$ is a finite-index subgroup of $I$, thus $\Gamma$ is a finite-index subgroup of $\Gamma'$.
	Thus, $\Gamma \backslash G / K = \Gamma \backslash G' / K'$ is a finite-sheeted cover of $M$.

	It remains to justify that $\Gamma$ is a cocompact lattice in $G$, equivalently, $\Gamma \backslash G$ is compact.
	Since $K$ is compact, this is equivalent to $\Gamma \backslash G / K$ being compact.
	This holds, since $\Gamma \backslash G / K$ is a finite-sheeted cover of the compact $M$.
	This completes the proof.
\end{proof}

From now on, we assume that $M = \Gamma \backslash G / K$, where $G$ is a connected Lie group, and
\begin{itemize}
	\item $K$ is a compact subgroup of $G$
	\item $\Gamma$ is a cocompact lattice of $G$
\end{itemize}
With this in mind, in order to prove Theorem~\ref{lhm}, the following Lemma about Lie groups will be crucial.
\begin{lem}\label{connected-lie-group-ses}
	Let $G$ be a connected Lie group. Then there is a short exact sequence of connected Lie groups
	\[
		1 \to H \to G \xra{p} \overline{G} \to 1
	\]
	where
	\begin{enumerate}
		%\item $H$ is a closed normal subgroup of $G$ that is
		\item $H$ is lattice hereditary (i.e. every lattice $\Gamma$ in $G$ yields a lattice $\Gamma \cap H$ as a lattice in $H$)
		\item $H$ is solvable-by-compact-by-abelian
		\item $\overline{G}$ is a semisimple Lie group without compact factors and without center.
	\end{enumerate}
\end{lem}
\begin{proof}
	We will define a sequence of surjective Lie group homomorphisms.
	\[
		p: G =: G_0
		\stackrel{p_1}{\thra} G_1
		\stackrel{p_2}{\thra} G_2
		\stackrel{p_3}{\thra} G_3 =: \overline{G}
	\]

	Let $G^{sol}$ be a solvradical (maximal connected solvable normal Lie subgroup) of $G$.
	By \cite{ragunathan:1972}, $G^{sol}$ exists, and is unique up to conjugation.
	Fix a choice of $G^{sol}$.
	A classical result (Levi decomposition) states that $G_1 := G/G^{sol}$ is semisimple.
	Let $p_1: G \to G_1$ be the canonical projection map,
	so that $\ker(p_1) = G^{sol}$ is solvable.

	Let $G_1^K$ be a compact radical (maximal connected compact normal Lie subgroup) of $G_1$.
	Then, $G_2 := G_1/G_1^K$ is a semisimple Lie group without compact factor.
	Let $p_2: G_1 \to G_2$ be the canonical projection map,
	so that $\ker(p_2) = G_1^K$ is compact.

	Let $G_2^Z$ be the center of $G_2$.
	By \cite[\S~3, Proposition~6.30]{knapp:1996}, $G_3 := G_2/G_2^Z$ is a semisimple Lie group without center and without compact factors.
	Let $p_3: G_2 \to G_3$ be the canonical projection map,
	so that $\ker(p_3) = G_2^Z$ is abelian.

	Now, let $p := p_3 \circ p_2 \circ p_1$, and $H := \ker(p)$.
	Then, by construction (c.f. Remark~\ref{group-by-extension-remark}), $\ker(p)$ is solvable-by-compact-by-abelian.
	Thus, combining the above, this proves 1 and 3.

	It remains to prove 2, i.e. $H = \ker(p)$ is lattice hereditary.
	To do so, we recall the following the following technical proposition:
	\begin{prop}[{\cite[Theorem~2.6]{geng:2015}}, {\cite{onishchik-vinberg:2000}}, {\cite[Theorem~1.13 (statement and proof)]{ragunathan:1972}}]\label{lattice_tfae}
		Let $G$ be a Lie group.
		Let $H < G$ be a closed subgroup, and let $\Gamma < G$ be lattice.
		If either $\Gamma$ is cocompact or $H$ is normal, then the following are equivalent:
		\begin{enumerate}
			%\item $H\Gamma$ is a closed subset of $G$.
			\item $\Gamma \cap H$ is a lattice in $H$ (in this case, we say that $H$ is \emph{$\Gamma$-hereditary}).
			\item $\Gamma \cap H$ is a cocompact lattice in $H$ (when $\Gamma$ is cocompact).
			%\item The image of $\Gamma$ in $G/H$ is discrete.
			\item The image of $\Gamma$ in $G/H$ is a lattice (when $H$ is normal).
			\item The image of $\Gamma$ in $G/H$ is a cocompact lattice (when $H$ is normal and $\Gamma$ is cocompact).
		\end{enumerate}
	\end{prop}
	Note that it will be used repeatedly, even in the main proof of Theorem~\ref{lhm}.

	To see this, let $\Lambda$ be any lattice in $G$.
	By the Bieberbach-Auslander-Wang Theorem{\cite[Proposition~2.5]{baues-kamishima:2018}}, $G^N := \ker(p_2 \circ p_1)$ is lattice hereditary in $G$.
	Thus, $\Lambda \cap G^N$ is a lattice in $G^N$.
	By Proposition~\ref{lattice_tfae}, $\Lambda_2 := p_2(p_1(\Lambda)) = \Lambda/(\Lambda \cap G^N)$ is a lattice in $N_2$.
	By {\cite[Proposition~5.18]{ragunathan:1972}}, $\ker(p_3) = G_2^Z$ is lattice hereditary in $G_2$.
	Thus, $\Lambda_2 \cap G_2^Z$ is a lattice in $G_2^Z$.
	By Proposition~\ref{lattice_tfae}, $p(\Lambda) = p_3(\Lambda_2)$ is a lattice in $G_3$.
	By Proposition~\ref{lattice_tfae} again, $\Lambda$ is a lattice in $H$, as desired.

	This completes the proof.
\end{proof}

\begin{rem}
	It might be tempting to use the Levi decomposition to justify Lemma~\ref{connected-lie-group-ses} directly.
	Unfortunately, a solvradical need not be lattice hereditary, even when $\Gamma$ is cocompact.
	To see this, consider $G = \R \times SO(3)$, so that $G^{sol} = \R$.
	Let $A$ an infinite order element in $SO(3)$ and $\Gamma := \rho(\Z)$ where
	\begin{align*}
		\rho: \Z &\to G \\
		n &\mapsto (n, A^n).
	\end{align*}
	Then $\Gamma$ is a cocompact lattice in $G$ where $\Gamma \cap S = \{0\}$, which is clearly not cocompact in $S$.
\end{rem}

We are now ready to use Lemma~\ref{connected-lie-group-ses} to prove Theorem~\ref{lhm}.

First, we would like to specify the base of the fiber bundle in Theorem~\ref{lhm}, which is supposed to be a closed locally symmetric space of non-compact type, as a double coset manifold given by the centerless semisimple $\overline{G}$.
With $\overline{K}^{\max}$ being any fixed maximal compact subgroup of $\overline{G}$ that contains $p(K)$, we have a symmetric space of non-compact type $\overline{G}/\overline{K}^{\max}$.
We would like to specify a cocompact lattice $\overline{\Gamma}'$ in $\overline{G}$ that acts freely on $\overline{G}/\overline{K}^{\max}$, which would yield a closed locally symmetric space of non-compact type $\overline{\Gamma}' \backslash \overline{G}/\overline{K}^{\max}$.
A natural candidate for $\overline{\Gamma}'$ is $\overline{\Gamma} := p(\Gamma) = \Gamma/(\Gamma \cap H)$.
However, a priori, there are two potential problems:
\begin{enumerate}
	\item $\overline{\Gamma}$ might not be a cocompact lattice (or even a discrete subgroup) in $\overline{G}$.
	\item Even if we assume that $\overline{\Gamma}$ is discrete,
	we can only guarantee that $\overline{\Gamma}$ acts on $\overline{G}/\overline{K}^{\max}$ with finite point-stabilizer (since the stabilizer of the point $\overline{g}\overline{K}^{\max} \in \overline{G}/\overline{K}^{\max}$ is $\overline{\Gamma}_{\overline{g}\overline{K}^{\max}} = \Gamma \cap \overline{g}\overline{K}^{\max}\overline{g}^{-1}$, which is discrete and compact, hence finite).
	In particular, $\overline{\Gamma}$ might not act freely on $\overline{G}/\overline{K}^{\max}$.
\end{enumerate}
Fortunately, it turns out that $\overline{\Gamma}$ is a virtually torsion-free, cocompact lattice in $\overline{G}$,
so that it admits a finite-index subgroup that satisfies the required conditions for $\overline{\Gamma}'$.
\begin{claim}\label{proof_base}
	Let $\overline{\Gamma}$ and $\overline{G}$ be as above.
	Then $\overline{\Gamma}$ admits a finite-index subgroup $\overline{\Gamma}'$ such that the double coset manifold $B = \overline{\Gamma}' \backslash \overline{G} / \overline{K}^{\max}$ is a closed locally symmetric space of non-compact type.
\end{claim}
\begin{proof}
	Since $\overline{G}/\overline{K}^{\max}$ is a symmetric space of non-compact type, to prove the claim, it suffices to show that $\overline{\Gamma}$ admits a finite-index subgroup $\overline{\Gamma}'$ is a torsion-free, cocompact lattice in $\overline{G} / \overline{K}^{\max}$, so that $\overline{\Gamma}' \backslash \overline{G} / \overline{K}^{\max}$ is a closed locally symmetric space of non-compact type.
	Since any finite-index subgroup of a cocompact lattice in $\overline{G}$ is still a finite-index subgroup of a cocompact lattice in $\overline{G}$, it suffices to prove that $\overline{\Gamma}$ is a virtually torsion-free, cocompact lattice in $\overline{G}$.
	To complete the proof, we prove the following facts about $\overline{\Gamma}$ accordingly:
	\begin{enumerate}
		\item $\overline{\Gamma}$ is virtually torsion-free:
			Since $\Gamma$ is the fundamental group of a compact manifold, it is finitely-generated.
			Since $\overline{\Gamma}$ is a quotient of $\Gamma$, it is also finitely-generated.
			By Lemma~\ref{connected-lie-group-ses}, $\overline{G}$ is a centerless semisimple Lie group, thus its adjoint map is injective, and thus it is linear.
			Thus, altogether, $\overline{\Gamma}$ is a finitely-generated linear group over field of characteristic-$0$, thus by a classical result of Selberg, it is virtually torsion-free.
		\item $\overline{\Gamma}$ is a cocompact lattice in $\overline{G}$:
			By Lemma~\ref{connected-lie-group-ses}, $H$ is lattice hereditary, thus $\Gamma \cap H$ is a lattice in $H$.
			By Proposition~\ref{lattice_tfae}, $\Gamma \cap H$ is a cocompact lattice in $H$.
			By Proposition~\ref{lattice_tfae} again, $\overline{\Gamma}$ is a cocompact lattice in $\overline{G}$.
	\end{enumerate}
	This completes the proof.
\end{proof}
To proceed, we fix a choice of $\overline{\Gamma}'$ as in Claim~\ref{proof_base}.
The resulting double coset manifold $B := \overline{\Gamma}' \backslash \overline{G} / \overline{K}^{\max}$ will be the base of the fiber bundle in Theorem~\ref{lhm}.

Next, we specify a finite-sheeted cover of $M$ that covers $B$, that will serve as the total space in Theorem~\ref{lhm}.
\begin{claim}
	Let $\Gamma' := p^{-1}(\overline{\Gamma}') \cap \Gamma$.
	The double coset manifold $M' := \Gamma' \backslash G / K$ is a finite cover of $M$.
\end{claim}
\begin{proof}
	By construction, $\Gamma'$ is a finite-index subgroup of $\Gamma$,
	thus $\Gamma' \backslash G / K$ is a finite cover of $\Gamma \backslash G / K = M$.
\end{proof}

We are now ready to specify the fiber bundle in the statement of Theorem~\ref{lhm}.
The choice of the projection map is the most obvious one (i.e. the canonical coset map).
The key point here is that it is indeed a fiber bundle.
\begin{claim}\label{fiber-bundle}
	Recall that $M' = \Gamma' \backslash G/K$ and $B = \overline{\Gamma}' \backslash \overline{G} / \overline{K}^{\max}$.
	Let $f: M' \to B$ be the (surjective) canonical coset (set-theoretic) map.
	Then $f$ is a fiber bundle with compact fiber.
\end{claim}
\begin{proof}
	By the Ehresmann fibration theorem, it suffices to prove that $f$ is a proper submersion.
	Then, since the total space if compact, it follows immediately that both the fiber and the base are compact.
	To complete the proof, we prove the following accordingly:
	\begin{enumerate}
		\item $f$ is a surjective submersion:
		This follows from the following, in logical order:
		\begin{enumerate}
			\item $p: G \to \overline{G}$ is a surjective submersion: this holds by the fact that $H$ is a closed subgroup of $G$ and Cartan's closed subgroup theorem
			\item $\overline{p}: G/K \to \overline{G}/\overline{K}^{\max}$ (with $\overline{p}$ being the canonical coset map) is a surjective submersion: this follows from the following facts:
				\begin{enumerate}
					\item $\overline{p}$ is smooth (as it descends from the smooth $p$) with respect to proper (right) actions by $K$ (resp. $\overline{K}^{\max}$)
					\item $\overline{p}$ is equivariant with respective to the transitive left $G$-action (c.f. \cite[Theorem~7.25 (equivariant rank theorem)]{lee:2012})
				\end{enumerate}
			\item $f: \Gamma' \backslash G / K \to \overline{\Gamma}' \backslash \overline{G} / \overline{K}^{\max}$ is a surjective submersion: this follows from the following facts:
				\begin{enumerate}
					\item $\overline{p}$ is $p$-equivariant (i.e. $l_{p(g)} \circ \overline{p} = \overline{p} \circ l_g$ for all $g \in G$)
					\item $\Gamma'$ (resp. $\overline{\Gamma}' = p(\Gamma')$) acts freely and properly discontinuously on $G/K$ (resp. $\overline{G} / \overline{K}^{\max}$)
				\end{enumerate}
		\end{enumerate}
		\item $\overline{\Gamma}' \backslash \overline{G} / \overline{K}^{\max}$ is compact: this follows from the fact that $\Gamma' \backslash G / K$ is compact and $f$ is surjective.
	\end{enumerate}
	This completes the proof.
\end{proof}

Finally, we show that the fiber of $f$ is connected, and is presented as a double coset manifold as in the statement of Theorem~\ref{lhm}.
\begin{claim}
 Let $F$ be the fiber of $f$.
 Then $F$ is diffeomorphic to $(\Gamma' \cap P) \backslash P / K$ where $P$ is a connected, (solvable-by-compact-by-abelian)-by-compact subgroup of $G$.
 (The connectedness of $P$ immediately implies the connectedness of $F$.)
\end{claim}
\begin{proof}
	By Claim~\ref{fiber-bundle}, $F$ is compact.

	To complete the proof, note that the fiber $F$ is homeomorphic to the preimage at the identity double coset point in the base $B$, i.e.
	\[
		F = f^{-1}(\Id)
		= (\Gamma' \cap P) \backslash P / K.
	\]
	where $P = p^{-1}(\overline{K}^{\max})$.

	Note that $P$ is the total space of the fiber bundle given by restricting the base of the fiber bundle (given by the short exact sequence of the underlying Lie groups)
	\[
		H \to G \to \overline{G}
	\]
	to $\overline{K}^{\max}$.
	Thus,
	\begin{enumerate}
		\item there exists a short exact sequence of Lie groups
		\[
			1 \to H \to P \to \overline{K}^{\max} \to 1.
		\]
		By Lemma~\ref{connected-lie-group-ses}, $H$ is solvable-by-compact-by-abelian.
		Since $\overline{K}^{\max}$ is compact, thus $P$ is (solvable-by-compact-by-abelian)-by-compact.
		\item Since $\overline{G}$ is homotopy equivalent to $\overline{K}^{\max}$, thus $P$ is homotopy equivalent to $G$, and hence is connected.
	\end{enumerate}
	This completes the proof.
\end{proof}

Since Claims 1-4 form the statement of Theorem~\ref{lhm}, this completes the proof of Theorem~\ref{lhm}.

\end{document}